\documentclass{amsart}
\usepackage{tikz}
\usepackage{amssymb,latexsym,pstricks}
\topskip  \numberwithin{equation}{section}
\newtheorem{theorem}{Theorem}[section]
\newtheorem{proposition}[theorem]{Proposition}
\newtheorem{corollary}[theorem]{Corollary}
\newtheorem{definition}[theorem]{Definition}

\newtheorem{remark}[theorem]{Remark}
\newtheorem{lemma}[theorem]{Lemma}

\newcommand*\circled[1]{\tikz[baseline=(char.base)]{
            \node[shape=circle,draw,inner sep=2pt] (char) {#1};}}
            
\begin{document}

\pagenumbering{arabic}
\pagestyle{headings}
\def\sof{\hfill\rule{2mm}{2mm}}
\def\llim{\lim_{n\rightarrow\infty}}

\title[Generalizations of Bell number formulas]{Generalizations of Bell number formulas of Spivey and Mez\H{o}}
\maketitle

\begin{center}Mark Shattuck\\
Department of Mathematics, University of Tennessee, Knoxville, TN 37919

{\tt shattuck@math.utk.edu}
\end{center}

%

\section*{Abstract}
We provide $q$-generalizations of Spivey's Bell number formula in various settings by considering statistics on different combinatorial structures.  This leads to new identities involving $q$-Stirling numbers of both kinds and $q$-Lah numbers.  As corollaries, we obtain identities for both binomial and $q$-binomial coefficients.  Our results at the same time also generalize recent $r$-Stirling number formulas of Mez\H{o}.  Finally, we provide a combinatorial proof and refinement of Xu's extension of Spivey's formula to the generalized Stirling numbers of Hsu and Shiue.  To do so, we develop a combinatorial interpretation for these numbers in terms of extended Lah distributions.

\noindent{Keywords}: Bell numbers, Stirling numbers, $q$-generalization, combinatorial proof

\noindent{2010 Mathematics Subject Classification}: 05A18, 05A19


\section{Introduction}

Let $S(n,k)$ denote the Stirling number of the second kind and $B(n)=\sum_{k=0}^n S(n,k)$, the $n$-th Bell number.  In 2008, Spivey \cite{Sp} proved the identity
\begin{equation}\label{Ieq1}
B(m+n)=\sum_{i=0}^n \sum_{j=0}^m j^{n-i}\binom{n}{i} S(m,j)B(i), \qquad m,n \geq 0,
\end{equation}
by an elegant combinatorial proof.  Since then several extensions have been given of \eqref{Ieq1} involving various generalizations of the Stirling and Bell numbers (see \cite{BM,GQ,Ka,Me,Xu}).

One such extension \cite{Me} involves the $r$-Stirling numbers \cite{Br} of both kinds and the $r$-Bell numbers \cite{Me0}, which count classes of permutations or partitions in which the elements $1,2,\ldots,r$ belong to distinct cycles or blocks.  Let $c^{(r)}(n,k)$ and $S^{(r)}(n,k)$ denote the $r$-Stirling numbers of the first and second kind, respectively, and $B^{(r)}(n)=\sum_{k=0}^nS^{(r)}(n,k)$ denote the $r$-Bell number.
Mez\H{o} \cite[Theorem 2]{Me} proved the following recursive formulas satisfied by these numbers.
\begin{theorem}[\textbf{Mez\H{o}}]\label{It1}
If $m,n,r\geq0$, then
\begin{equation}\label{It1e1}
B^{(r)}(m+n)=\sum_{i=0}^n\sum_{j=0}^m (j+r)^{n-i}\binom{n}{i}S^{(r)}(m,j)B(i)
\end{equation}
and
\begin{equation}\label{It1e2}
(r+1)^{\overline{m+n}}=\sum_{i=0}^n \sum_{j=0}^m m^{\overline{n-i}}\binom{n}{i}c^{(r)}(m,j)(r+1)^{\overline{i}}.
\end{equation}
\end{theorem}

Here, $a^{\overline{b}}=a(a+1)\cdots (a+b-1)$ if $b \geq 1$ denotes the \emph{rising factorial}, with $a^{\overline{0}}=1$.  Note that \eqref{It1e1} reduces to \eqref{Ieq1} when $r=0$.

In the next section, we provide $q$-generalizations of formulas \eqref{It1e1} and \eqref{It1e2} by considering statistics originally due to Carlitz \cite{Ca2} on set partitions and permutations and restricting those statistics to the substructures enumerated by the $r$-Stirling numbers of the first and second kinds (see \eqref{P1e2} and \eqref{t4c1e1} below).  Some special cases of these identities are noted.  We also consider an analogous $r$-Lah number (see \cite{NR}) that counts the class of Lah distributions (i.e., set partitions where the order matters within a block, for example, see \cite{LMS}) in which the elements $1,2,\ldots,r$ belong to distinct blocks.  By considering an earlier statistic of Garsia and Remmel \cite{GR}, we obtain $q$-extensions of an $r$-Lah number analogue of identities \eqref{It1e1} and \eqref{It1e2}.  This leads to an apparently new identity involving $q$-binomial coefficients.

Xu \cite{Xu} considered another type of extension of Spivey's result in terms of the generalized Stirling numbers of Hsu and Shiue \cite{HS}, which are denoted by $S(n,k;\alpha,\beta,r)$.  It was shown in \cite{Xu} by algebraic methods that
\begin{equation}\label{Ieq2}
B_{m+n;\alpha,\beta,r}(x)=\sum_{i=0}^n \sum_{j=0}^m \binom{n}{i}x^jS(m,j;\alpha,\beta,r)B_{i;\alpha,\beta,r}(x)\prod_{\ell=0}^{n-i-1}((m+\ell)\alpha+j\beta), \qquad m,n \geq 0,
\end{equation}
where $B_{n;\alpha,\beta,r}(x)=\sum_{k=0}^n S(n,k;\alpha,\beta,r)x^k$.  The numbers $S(n,k;\alpha,\beta,r)$ are seen to generalize a variety of combinatorial arrays, including Stirling numbers of both kinds, Lah numbers, and binomial coefficients.  For example, Spivey's formula \eqref{Ieq1} would correspond to the $x=\beta=1$ and $\alpha=r=0$ case of \eqref{Ieq2}.  Here, we develop a combinatorial interpretation of the numbers $S(n,k;\alpha,\beta,r)$, which have primarily been considered from an algebraic standpoint, as weighted sums over an extension of the set of Lah distributions.  Using this interpretation, we are then able to show that there is a one-to-one correspondence between structures enumerated by both sides of \eqref{Ieq2}.

We now recall some basic terminology and notation.  We will follow the convention that empty sums assume the value zero and empty products, the value one.  By a \emph{partition} of the set $[n]=\{1,2,\ldots,n\}$, we will mean a collection of non-empty, disjoint subsets, called \emph{blocks}, whose union is $[n]$ (see, e.g., \cite{Ma}).  The number of partitions of $[n]$ is given by $B(n)$, with $S(n,k)$ counting those having exactly $k$ blocks.  If $m$ and $n$ are positive integers, then let $[m,n]=\{m,m+1,\ldots,n\}$ if $m \leq n$, with $[m,n]=\varnothing$ if $m>n$.  If $q$ is an indeterminate, then $n_q=1+q+\cdots+q^{n-1}$ if $n \geq 1$, with $0_q=0$.  If $n \geq 1$, then the $q$-factorial is defined as $n_q!=\prod_{i=1}^n i_q$ if $n \geq 1$, with $0_q!=1$.  The $q$-binomial coefficient $\binom{n}{k}_q$ is given by $\frac{n_q!}{k_q!(n-k)_q!}$ if $0 \leq k \leq n$, with $\binom{n}{k}_q=1$ if $k=0$ for all integers $n$, and will be assumed to be zero otherwise.

\section{$q$-analogues of Spivey's formula}

In this section, we find $q$-analogues of identity \eqref{Ieq1} and the formulas in Theorem \ref{It1} which come about from considering certain statistics on set partitions, Lah distributions, and permutations.

\subsection{$q$-Stirling identities of the second kind}

We will need the following terminology and notation.  Let $S_q(n,k)$ (see, e.g., \cite{Wa}) be defined by the recurrence
$$S_q(n,k)=q^{k-1}S_q(n-1,k-1)+k_qS_q(n-1,k), \qquad n,k \geq 1,$$
with $S_q(n,k)=\delta_{n,k}$ if $n=0$ or $k=0$.  Note that $S_q(n,k)=q^{\binom{k}{2}}\widetilde{S}_q(n,k)$, where $\widetilde{S}_q(n,k)$ was considered by Carlitz \cite{Ca0,Ca1}.  Let $\mathcal{P}_{n,k}$ denote the set of partitions of $[n]$ into $k$ blocks and $\mathcal{P}_n=\cup_{k=0}^n\mathcal{P}_{n,k}$, the set of all partitions of $[n]$.  Given a partition $\pi=B_1/B_2/\cdots/B_k \in \mathcal{P}_{n,k}$ such that $\min B_1<\min B_2<\cdots <\min B_k$, let $w(\pi)=\sum_{i=1}^k (i-1)|B_i|$ (see, e.g., \cite{Ca2,Ma,Wa}).  It follows from the defining recurrence for $S_q(n,k)$, upon considering the position of the element $n$ within a member of $\mathcal{P}_{n,k}$, that
$$S_q(n,k)=\sum_{\pi \in \mathcal{P}_{n,k}}q^{w(\pi)}, \qquad n,k \geq 0.$$  If $n \geq 0$, then let $B_q(n)=\sum_{k=0}^n S_q(n,k)$ denote the corresponding $q$-Bell number \cite{Ma,Wa}.

Given $r \geq0$, let $\mathcal{P}^{(r)}_{n,k}$ denote the set of partitions of $[n+r]$ into $k+r$
blocks in which the elements $1,2,\ldots,r$ belong to distinct blocks and let $\mathcal{P}_n^{(r)}=\cup_{k=0}^n \mathcal{P}^{(r)}_{n,k}$.  The respective cardinalities of $\mathcal{P}^{(r)}_{n,k}$
and $\mathcal{P}^{(r)}_{n}$ are known as the $r$-Stirling and $r$-Bell numbers (see \cite{Br,Me0}) and are denoted here by $S^{(r)}(n,k)$ and $B^{(r)}(n)$, respectively.

We now $q$-generalize the numbers $S^{(r)}(n,k)$ and $B^{(r)}(n)$ by letting
$$S_q^{(r)}(n,k)=\sum_{\pi \in \mathcal{P}^{(r)}_{n,k}} q^{w(\pi)}, \qquad n,k \geq 0,$$
and $B_q^{(r)}(n)=\sum_{k=0}^n S_q^{(r)}(n,k)$.  Note that $S_q^{(r)}(n,k)$ and $B_q^{(r)}(n)$ reduce to $S_q(n,k)$ and $B_q(n)$ when $r=0$.

Considering the number, $i$, of elements of $[r+1,n+1]$ that do not lie in a block containing an element of $[r]$ within a member of $\mathcal{P}^{(r)}_{n,k}$ yields the relation
\begin{equation}\label{Pe1}
S_q^{(r)}(n,k)=\sum_{i=0}^n q^{ir}r_q^{n-i}\binom{n}{i} S_q(i,k), \qquad r \geq 0.
\end{equation}

The sequences $S_q^{(r)}(n,k)$ and $B_q^{(r)}(n)$ satisfy the following recurrences.  In this proof and those in later sections, we will denote the set $[m+r+1,m+n+r]$ by $I$.

\begin{theorem}\label{P1}
If $m,n,r \geq 0$, then
\begin{equation}\label{P1e1}
S_q^{(r)}(m+n,k)=\sum_{i=0}^n \sum_{j=0}^m q^{i(j+r)}[j+r]_q^{n-i}\binom{n}{i} S_q^{(r)}(m,j)S_q(i,k-j)
\end{equation}
and
\begin{equation}\label{P1e2}
B_q^{(r)}(m+n)=\sum_{i=0}^n \sum_{j=0}^mq^{i(j+r)}[j+r]_q^{n-i}\binom{n}{i} S_q^{(r)}(m,j) B_q(i).
\end{equation}
\end{theorem}
\begin{proof}
The second identity follows from the first by summing over $k$ or by allowing partitions to contain any number of blocks in what follows.  To show \eqref{P1e1}, we will argue that the $(i,j)$ term of the sum gives the weight (with respect to the $w$ statistic) of all members of $\mathcal{P}^{(r)}_{m+n,k}$ in which the elements of $[m+r]$ occupy exactly $j+r$ blocks, with $n-i$ elements of $I$ also belonging to those blocks.  Note that within such partitions, there are $S_q^{(r)}(m,j)$ possibilities (upon briefly regarding the indeterminate $q$ as a positive integer) for the partition comprising the elements of $[m+r]$ and $\binom{n}{n-i}$ choices for the elements of $I$ that are to go in this partition.  The factor $[j+r]_q^{n-i}$ then accounts for the placement of these
$n-i$ elements as there are $1+q+\cdots+q^{j+r-1}=[j+r]_q$ possibilities for each element (insertion into the $\ell$-th block from the left gives a factor of $q^{\ell-1}$ for each $1 \leq \ell \leq j+r$).  The remaining $i$ elements are then to be partitioned into $k-j$ additional blocks. Note that each of these blocks is translated to the right by $j+r$ positions, which implies that there are $q^{i(j+r)}S_q(i,k-j)$ possibilities for the remaining $i$ elements.
\end{proof}

\begin{remark}  Letting $r=0$ in \eqref{P1e2} recovers the identity of Katriel \cite{Ka}
(shown by operator methods) and thus we have obtained a combinatorial proof for it.  Letting $q=1$ in \eqref{P1e2} gives formula \eqref{It1e1} of Mez\H{o}.
\end{remark}

If $P$ is a statement, then let $\chi(P)=1$ or $0$ depending on whether $P$ is true or false. Given $n\geq1$, let $$\alpha_{i,j}=\chi(j \text{~is~odd})+\chi(j \text{~is~even~and~} i=n)$$ and $$\beta_{i,j}=\chi(j \text{~is~even})+\chi(j \text{~is~odd~and~} i=n)$$
for positive integers $i$ and $j$.  We now state some binomial coefficient identities which follow from the prior theorem.

\begin{proposition}\label{bin}
If $m,n,k \geq 1$, then
\begin{equation}\label{bine1}
\binom{m+n-k-1}{k-1}=\sum_{i=0}^n\sum_{j=0}^m \alpha_{i,j}(-1)^{(i+1)j}\binom{n}{i}\binom{m-\lfloor j/2\rfloor-1}{m-j}\binom{i-k+\lceil j/2\rceil-1}{i-2k+j},
\end{equation}
\begin{equation}\label{bine2}
\binom{m+n-k}{k-1}=\sum_{i=0}^n\sum_{j=0}^m \alpha_{i,j}(-1)^{ij}\binom{n}{i}\binom{m-\lfloor j/2\rfloor-1}{m-j}\binom{i-k+\lfloor j/2\rfloor}{i-2k+j+1},
\end{equation}
\begin{equation}\label{bine3}
\binom{m+n-k}{k}=\sum_{i=0}^n\sum_{j=0}^m \beta_{i,j}(-1)^{i(j+1)}\binom{n}{i}\binom{m-\lceil j/2\rceil}{m-j}\binom{i-k+\lceil j/2\rceil-1}{i-2k+j},
\end{equation}
and
\begin{equation}\label{bine4}
\binom{m+n-k}{k-1}=\sum_{i=0}^n\sum_{j=0}^m \beta_{i,j}(-1)^{(i+1)(j+1)}\binom{n}{i}\binom{m-\lceil j/2\rceil}{m-j}\binom{i-k+\lfloor j/2\rfloor}{i-2k+j+1}.
\end{equation}
\end{proposition}
\begin{proof}
To show \eqref{bine1} and \eqref{bine2}, we take $q=-1$ and $r=0$ in \eqref{P1e1}. (Indeed, one gets the same result if one takes any even $r$ as the evaluations at $q=-1$ are all the same, by \eqref{Pe1}.) Since $m_q\mid_{q=-1}=\chi(m \text{~is~odd})$ and since the factor $[j+r]_q^{n-i}$  is $1$ for all $j$ when $i=n$, we have
\begin{equation}\label{bine5}
S_{-1}(m+n,k)=\sum_{i=0}^n\sum_{j=0}^m \alpha_{i,j}(-1)^{ij}\binom{n}{i}S_{-1}(m,j)S_{-1}(i,k-j),
\end{equation}
where $S_{-1}(n,k)=S_q(n,k)\mid_{q=-1}$. From \cite[Theorem 4.1]{Wa}, it follows that
\begin{equation}\label{bine6}
S_{-1}(n,k)=(-1)^{\binom{k}{2}}\binom{n-\lfloor k/2\rfloor -1}{n-k}, \qquad 0 \leq k \leq n.
\end{equation}
(The reader is referred to \cite{Sh} for a combinatorial proof.)  Identities \eqref{bine1} and \eqref{bine2} now follow from \eqref{bine5} and \eqref{bine6} by replacing $k$ with $2k$ and $2k-1$, respectively, and noting the fact $\binom{\ell}{2}=\binom{j}{2}+\binom{\ell-j}{2}+j(\ell-j)$ for $0 \leq j \leq \ell$.

To show \eqref{bine3} and \eqref{bine4}, we first take $q=-1$ and $r=1$ in \eqref{P1e1} (in fact, one could take any odd $r$).  This gives
\begin{equation}\label{bine7}
S_{-1}^{(1)}(m+n,k)=\sum_{i=0}^n \sum_{j=0}^m \beta_{i,j}(-1)^{i(j+1)}S_{-1}^{(1)}(m,j)S_{-1}(i,k-j).
\end{equation}
We seek an expression for $S_{-1}^{(1)}(n,k)$.  By \eqref{Pe1} at $q=-1$ and \eqref{bine6}, we have
\begin{equation}\label{bine8}
S_{-1}^{(1)}(n,k)=\sum_{i=k}^n (-1)^{i+\binom{k}{2}}\binom{n}{i}\binom{i-\lfloor k/2\rfloor-1}{i-k}.
\end{equation}

To find $S_{-1}^{(1)}(n,k)$, we make use of the generating function method as described in Wilf \cite[Section 4.3]{Wi}.  Let us fix $k$ and calculate the generating function $f_k(x):=\sum_{n\geq k}a_{n,k}x^n$, where $a_{n,k}=(-1)^{\binom{k}{2}}S_{-1}^{(1)}(n,k)$.  A computation gives
$$f_k(x)=\frac{(-x)^k}{(1-x)^{\lfloor k/2\rfloor+1}}.$$
Thus, we have
$$\sum_{n\geq k}a_{n,k}x^n=(-1)^k x^{\lceil k/2 \rceil}\cdot \frac{x^{\lfloor k/2 \rfloor}}{(1-x)^{\lfloor k/2\rfloor+1}}=(-1)^k\sum_{n\geq k}\binom{n-\lceil k/2\rceil}{\lfloor k/2 \rfloor}x^n,$$
which implies
\begin{equation}\label{bine9}
S_{-1}^{(1)}(n,k)=(-1)^{\binom{k}{2}}a_{n,k}=(-1)^{\binom{k+1}{2}}\binom{n-\lceil k/2\rceil}{\lfloor k/2 \rfloor}.
\end{equation}
Identities \eqref{bine3} and \eqref{bine4} now follow from \eqref{bine6}, \eqref{bine7} and \eqref{bine9}, upon replacing $k$ by $2k$ and $2k-1$, respectively.
\end{proof}

\subsection{$q$-Lah identities}

A partition of $[n]$ in which the order of the elements within a block matters is called a \emph{Lah distribution}.  Let $\mathcal{L}_{n,k}$ denote the set of all Lah distributions having $k$ blocks.  The cardinality of $\mathcal{L}_{n,k}$ is called the Lah number $L(n,k)$ (see \cite{La}).

We define a statistic on $\mathcal{L}_{n,k}$ as follows.  Given $\delta \in \mathcal{L}_{n,k}$, represent the contents of each ordered block by a word and then arrange these words in a sequence $W_1,W_2,\ldots,W_k$ by decreasing order of their least elements.  Then replace the commas in this sequence by zeros and count inversions in the resulting single word to obtain the value $\text{inv}_\rho(\delta)$, i.e.,
$$\text{inv}_\rho(\delta)=\text{inv}(W_10W_20\cdots W_{k-1}0W_k).$$
For example, if $\delta=\{3,2,5\},\{7,6,8\},\{1,4\} \in \mathcal{L}_{n,k}$, then we have $\text{inv}_\rho(\delta)=30$, the number of inversions in the word $7680325014$.

The statistic $\text{inv}_\rho$ is due to Garsia and Remmel \cite{GR}, who showed
\begin{equation}\label{lahe1}
L_q(n,k):=\sum_{\delta \in \mathcal{L}_{n,k}}q^{\text{inv}_\rho(\delta)}=q^{k(k-1)}\frac{n_q!}{k_q!}\binom{n-1}{k-1}_q, \qquad 1 \leq k \leq n,
\end{equation}
which generalizes the well known formula for $L(n,k)$.  For other examples of statistics on $\mathcal{L}_{n,k}$, see, e.g., \cite{LMS}.  Considering the position of the element $n$ within a member of $\mathcal{L}_{n,k}$, one obtains the recurrence
$$L_q(n,k)=q^{n+k-2}L_q(n-1,k-1)+[n+k-1]_qL_q(n-1,k), \qquad n,k \geq 1,$$
with $L_q(n,k)=\delta_{n,k}$ if $n=0$ or $k=0$.

We now recall a Lah number analogue of the $r$-Stirling numbers (see \cite{NR}), for which we will provide a $q$-generalization.  Let $\mathcal{L}^{(r)}_{n,k}$ denote the set of Lah distributions of $[n+r]$ into $k+r$ blocks in which the elements $1,2,\ldots,r$ belong to distinct blocks and let $L^{(r)}(n,k)=|\mathcal{L}^{(r)}_{n,k}|$.  Considering the number, $i$, of elements in $[r+1,r+n]$ that occupy a block containing a member of $[r]$ yields the relation
$$L^{(r)}(n,k)=\sum_{i=0}^n \prod_{\ell=0}^{i-1}(\ell+2r)\binom{n}{i}L(n-i,k), \qquad r \geq 0.$$
Let $L^{(r)}(n)=\sum_{k=0}^rL^{(r)}(n,k)$.

We $q$-generalize the numbers $L^{(r)}(n,k)$ and $L^{(r)}(n)$ by letting
$$L^{(r)}_q(n,k)=\sum_{\delta\in\mathcal{L}^{(r)}_{n,k}}q^{\text{inv}_\rho(\delta)}, \qquad n,k \geq 0,$$
and $L_q^{(r)}(n)=\sum_{k=0}^nL^{(r)}_q(n,k)$.  Note that for $n \geq 0$, we have $L^{(r)}_q(n,k)=0$ if $k>n$ or $k<0$.

Define the $q$-\emph{rising factorial} by $[n]_q^{\overline{m}}=\prod_{i=n}^{n+m-1} i_q$ if $m \geq 1$, with $[n]_q^{\overline{0}}=1$ for all $n \geq0$.  Recall also that the $q$-binomial coefficient $\binom{n}{k}_q$ is the distribution function on the set of binary words consisting of $k$ zeros and $n-k$ ones for the statistic which records the number of inversions \cite[Prop. 1.3.17]{St}.

The sequences $L^{(r)}_q(n,k)$ and $L_q^{(r)}(n)$ satisfy the following relations.

\begin{theorem}\label{P2}
If $m,n,r \geq 0$, then
\begin{equation}\label{P2e1}
L^{(r)}_q(m+n,k)=\sum_{i=0}^n \sum_{j=0}^k q^{i(j+m+2r)}[j+m+2r]_q ^{\overline{n-i}}\binom{n}{i}_qL_q^{(r)}(m,j)L_q(i,k-j)
\end{equation}
and
\begin{equation}\label{P2e2}
L^{(r)}_q(m+n)=\sum_{i=0}^n \sum_{j=0}^k q^{i(j+m+2r)}[j+m+2r]_q ^{\overline{n-i}}\binom{n}{i}_qL_q^{(r)}(m,j)L_q(i).
\end{equation}
\end{theorem}
\begin{proof}
The second identity follows from the first by summing over $k$.  To show \eqref{P2e1}, suppose that $\delta \in \mathcal{L}^{(r)}_{m+n,k}$ is such that the elements of $[m+r]$ occupy exactly $j+r$ blocks, with $n-i$ elements of $I$ also lying in these blocks.  Then there are $L_q^{(r)}(m,j)$ possibilities concerning the placement of the elements of $[m+r]$ and $L_q(i,k-j)$ possibilities for the remaining $i$ elements of $I$.  Note that each of these elements creates an inversion with every element of $[m+r]$, by the ordering of the blocks, as well as an inversion with each zero that precedes a block containing a member of $[m+r]$.  Thus, there are $(m+r)+(j+r)=j+m+2r$ inversions in all created by each of these $i$ elements, which accounts for the $q^{i(j+m+2r)}$ factor.

Let $S=\{a_1<a_2<\cdots<a_{n-i}\}$ denote the subset of $I$ whose elements belong to a block of $\delta$ containing at least one member of $[m+r]$.  We insert the elements of $S$ into the subpartition comprising the members of $[m+r]$, one at a time in increasing order, starting with $a_1$.  Note that there are $1+q+\cdots+q^{j+m+2r-1}=[j+m+2r]_q$ possibilities concerning the placement of $a_1$ since it can be inserted anywhere amongst a sequence consisting of the letters in $[m+r]$ together with $j+r-1$ zeros, with the choice of position creating anywhere from zero to $j+m+2r-1$ inversions.  By similar reasoning, the factor $[j+m+2r+\ell]_q$ accounts for the choice of position of the element $a_{\ell+1}$ for each $0 \leq \ell \leq n-i-1$.

Finally, observe that the number of inversions of $\delta$ (as defined by the $\text{inv}_\rho$ statistic) involving a member of $S$ and a member of $I-S$ is the same as the number of inversions (as defined in the usual sense) in the binary word $w=w_1w_2\cdots w_n$ obtained by setting $w_\ell=1$ if $m+r+\ell \in S$ and $w_\ell=0$ if $m+r+\ell \not \in S$.  This accounts for the $\binom{n}{i}_q$ factor and completes the proof.
\end{proof}

As a consequence, we obtain the following $q$-binomial coefficient identity.

\begin{corollary}\label{P2c1}
If $m,n,k \geq 0$, then
\begin{align}
\binom{m+n}{k}_q&\binom{m+n+1}{n}_q-\binom{m}{k}_q\binom{k+m+n+1}{n}_q\notag\\
&=\sum_{i=1}^n\sum_{j=1}^{k}q^{i(j+m+1)-2j(k-j+1)}\binom{m}{j-1}_q\binom{k+1}{j}_q\binom{i-1}{k-j}_q\binom{m+n+j-i}{n-i}_q.\label{c1e1}
\end{align}
\end{corollary}
\begin{proof}
Applying \eqref{lahe1} to both sides of \eqref{P2e1} when $r=0$ and $m,n,k \geq 1$ (and separating off the $i=0$, $j=k$ term in the sum on the right-hand side) gives
\begin{align}
q^{k(k-1)}&\frac{(m+n)_q!}{k_q!}\binom{m+n-1}{k-1}_q-q^{k(k-1)}\frac{m_q!(k+m+n-1)_q!}{k_q!(k+m-1)_q!}\binom{m-1}{k-1}_q\notag\\
&=\sum_{i=1}^n\sum_{j=1}^{k-1}q^{i(j+m)+2\binom{j}{2}+2\binom{k-j}{2}}\frac{i_q!m_q!(j+m+n-i-1)_q!}{j_q!(k-j)_q!(j+m-1)_q!}\binom{n}{i}_q\binom{m-1}{j-1}_q\binom{i-1}{k-j-1}_q.\notag
\end{align}
Multiply this last equation by $\frac{k_q!}{m_q!n_q!}$, replace $\binom{n}{i}_q$ with $\frac{n_q!}{i_q!(n-i)_q!}$, and rearrange factors on both sides.  The desired result now follows from noting $\binom{j}{2}+\binom{k-j}{2}=\binom{k}{2}-j(k-j)$, canceling $q^{k(k-1)}$ factors on both sides, and finally replacing $k$ with $k+1$ and $m$ with $m+1$.
\end{proof}

\subsection{$q$-Stirling number identities of the first kind}

Let $\mathcal{G}_{n,k}$ denote the set of permutations of $[n]$ having $k$ cycles and $\mathcal{G}_n$ the set of all permutations of $[n]$.  The cardinality of $\mathcal{G}_{n,k}$ is the (signless) Stirling number of the first kind, often denoted $c(n,k)$.  In what follows, we will represent $\pi=C_1/C_2/\cdots/C_k \in \mathcal{G}_{n,k}$ in \emph{standard cycle form}, i.e., $\min C_1<\min C_2<\cdots <\min C_k$, with the smallest element written first within each cycle.

We recall a statistic on $\mathcal{G}_{n,k}$ originally considered by Carlitz \cite{Ca2} and later studied \cite{Sh2}.  First express $\pi=C_1/C_2/\cdots/C_k \in \mathcal{G}_{n,k}$ in standard cycle form.  Then erase the internal dividers and count inversions in the resulting word.  Denote the value so obtained by $\text{inv}_c(\pi)$.
For example, if $\pi \in \mathcal{G}_{7,3}$ has cycles of $(5,7,6)$, $(3,4,1)$ and $(2)$, then $\text{inv}_c(\pi)=3$, the number of inversions in the word $1342576$.
Let $$c_q(n,k)=\sum_{\pi \in \mathcal{G}_{n,k}}q^{\text{inv}_c(\pi)}, \qquad n,k \geq 0.$$
Considering the position of the element $n$ relative to the members of $[n-1]$ gives the recurrence
\begin{equation}\label{cqe1}
c_q(n,k)=c_q(n-1,k-1)+[n-1]_qc_q(n-1,k), \qquad n,k \geq 1,
\end{equation}
with $c_q(n,k)=\delta_{n,k}$ if $n=0$ or $k=0$.

Let $\mathcal{G}_{n,k}^{(r)}$ denote the set of permutations of $[n+r]$ having $k+r$ cycles in which the elements $1,2,\ldots,r$ belong to distinct cycles and let $\mathcal{G}_n^{(r)}=\cup_{k=0}^n\mathcal{G}_{n,k}^{(r)}$.  The cardinality of $\mathcal{G}_{n,k}^{(r)}$ is given by the $r$-Stirling number of the first kind \cite{Br}, which we denote here by $c^{(r)}(n,k)$.  We now consider a $q$-generalization of $c^{(r)}(n,k)$ obtained by letting
$$c_q^{(r)}(n,k)=\sum_{\pi \in \mathcal{G}_{n,k}^{(r)}}q^{\text{inv}_c(\pi)}, \qquad n,k \geq 0.$$

The sequence $c_q^{(r)}(n,k)$ satisfies the following recurrence relation.

\begin{theorem}\label{t3}
If $m,n,r \geq 0$, then
\begin{equation}\label{t3e1}
c_q^{(r)}(m+n,k)=\sum_{i=0}^n \sum_{j=0}^m [m+r]_q^{\overline{n-i}}\binom{n}{i}_q c_q^{(r)}(m,j)c_q(i,k-j)
\end{equation}
and
\begin{equation}\label{t3e2}
\prod_{\ell=0}^{n-1}(1+[\ell+m+r]_q)=\sum_{i=0}^n [m+r]_q^{\overline{n-i}}\binom{n}{i}_q \prod_{\ell=0}^{i-1}(1+\ell_q).
\end{equation}
\end{theorem}
\begin{proof}
For \eqref{t3e1}, we proceed as in prior proofs and consider the number, $n-i$, of elements of $I$ that occupy a cycle containing a member of $[m+r]$.
There are $c_q^{(r)}(m,j)$ possibilities for the placement of the elements of $[m+r]$ for some $j$ and $c_q(i,k-j)$ possibilities for the elements of $I$ not belonging to a cycle containing a member of $[m+r]$.  Note that there are no inversions between these elements of $I$ and elements of $[m+r]$, by the ordering of the cycles.  The factors $[m+r]_q^{\overline{n-i}}$ and $\binom{n}{i}_q$ arise in a similar manner as in the proof of \eqref{P2e1} above, which completes the proof of \eqref{t3e1}.

Next observe that
$$\sum_{k=0}^n c_q^{(r)}(n,k)=\prod_{\ell=r}^{n+r-1}(1+\ell_q), \qquad r \geq 0.$$
This can be seen combinatorially by considering the positions of elements of $[r+1,r+n]$ within a member of $\mathcal{G}_n^{(r)}$.  Note that there are always $1+\ell_q$ possibilities for the position of the element $\ell+1$ for $r \leq \ell \leq r+n-1$ given any permissible arrangement of the members of $[\ell]$.  Thus summing \eqref{t3e1} over $k$ gives
\begin{align}
\prod_{\ell=r}^{m+n+r-1}(1+\ell_q)&=\sum_{i=0}^n\sum_{j=0}^m [m+r]_q^{\overline{n-i}}\binom{n}{i}_q c_q^{(r)}(m,j)\prod_{\ell=0}^{i-1}(1+\ell_q)\notag\\
&=\sum_{i=0}^n[m+r]_q^{\overline{n-i}}\binom{n}{i}_q\prod_{\ell=0}^{i-1}(1+\ell_q)\sum_{j=0}^mc_q^{(r)}(m,j)\notag\\
&=\prod_{\ell=r}^{m+r-1}(1+\ell_q)\sum_{i=0}^n[m+r]_q^{\overline{n-i}}\binom{n}{i}_q\prod_{\ell=0}^{i-1}(1+\ell_q),\notag
\end{align}
which implies \eqref{t3e2}.  Alternatively, using prior reasoning, one can show directly that both sides of \eqref{t3e2} give the total weight with respect to the $\text{inv}_c$ statistic of all the members of $\mathcal{G}_n^{(m+r)}$.
\end{proof}

Using \eqref{cqe1}, one can show that $c_q(n,k)$ is the $(n-k)$-th symmetric function in the quantities $\{1_q,2_q,\ldots,[n-1]_q\}$.  Substituting this into \eqref{t3e1} gives an explicit recurrence for $c_q^{(r)}(a,k)$ in terms of $c_q^{(r)}(b,j)$ for $b<a$ and $j\leq k$ for a fixed $r$.

\subsection{Further $q$-identities}

We have the following additional recurrence relations satisfied by the sequences of the prior sections.

\begin{theorem}\label{t4}
If $m,n,r \geq 0$, then
\begin{equation}\label{t4e1}
S_q^{(m+r)}(n,k)=\sum_{i=0}^n q^{m(i+r)+\binom{m}{2}}m_q^{n-i}\binom{n}{i}S_q^{(r)}(i,k),
\end{equation}
\begin{equation}\label{t4e2}
L_q^{(m+r)}(n,k)=\sum_{i=0}^n q^{m(2i+2r+m-1)}[2m]_q^{\overline{n-i}}\binom{n}{i}_qL_q^{(r)}(i,k),
\end{equation}
and
\begin{equation}\label{t4e3}
c_q^{(m+r)}(n,k)=\sum_{i=0}^n q^{r(n-i)}m_q^{\overline{n-i}}\binom{n}{i}_qc_q^{(r)}(i,k).
\end{equation}
\end{theorem}
\begin{proof}
The proofs of these identities are similar, so we only show \eqref{t4e3} and leave the other two as exercises.  We argue that the right-hand side of \eqref{t4e3} gives the total weight with respect to the $\text{inv}_c$ statistic of all the members of $\mathcal{G}_{n,k}^{(m+r)}$.  Consider the number, $n-i$, of elements of $I$ which belong to a cycle containing a member of $[m]$.  For these elements, there are $m_q^{\overline{n-i}}$ possibilities.  The remaining $i$ elements of $I$, when taken together with $[m+1,m+r]$, comprise a member of $\mathcal{G}_{i,k}^{(r)}$ and thus there are $c_q^{(r)}(i,k)$ possibilities concerning their placement. The factor $\binom{n}{i}_q$ arises in much the same way as before since the $n-i$ letters of $I$ belonging to the first $m$ cycles can create inversions with letters of $I$ that do not.  Furthermore, these $n-i$ letters of $I$ create $r(n-i)$ additional inversions with the members of $[m+1,m+r]$, by the ordering of cycles, whence the factor of $q^{r(n-i)}$. Summing over all possible $i$ gives \eqref{t4e3}.
\end{proof}

Further identities may be obtained by summing those in the preceding theorem over $k$.  For example, summing both sides of \eqref{t4e3} over $k$ gives
$$\prod_{\ell=0}^{n-1}(1+[\ell+m+r]_q)=\sum_{i=0}^nq^{r(n-i)}m_q^{\overline{n-i}}\binom{n}{i}_q\prod_{\ell=0}^{i-1}(1+[\ell+r]_q),$$
which may be rewritten, upon multiplying both sides by $\prod_{\ell=0}^{m-1}(1+[\ell+r]_q)$, as follows.

\begin{corollary}\label{t4c1}
If $m,n,r \geq 0$, then
\begin{equation}\label{t4c1e1}
\prod_{\ell=0}^{m+n-1}(1+[\ell+r]_q)=\sum_{i=0}^n\sum_{j=0}^mq^{r(n-i)}m_q^{\overline{n-i}}\binom{n}{i}_qc_q^{(r)}(m,j)\prod_{\ell=0}^{i-1}(1+[\ell+r]_q).
\end{equation}
\end{corollary}

\begin{remark}
Taking $q=1$ in \eqref{t4c1e1} recovers formula \eqref{It1e2} of Mez\H{o}.
\end{remark}

\section{Combinatorial proof of a prior recurrence}

Xu \cite{Xu} extended Spivey's recurrence to the generalized Stirling numbers of Hsu and Shiue \cite{HS}.  Note that these Stirling numbers have as special cases several earlier generalizations, including those of Carlitz \cite{Ca3,Ca4}, Gould and Hopper \cite{GH}, and Howard \cite{Ho}.  The proof given by Hsu and Shiue is algebraic and makes careful use of exponential generating functions.  Here, we provide a combinatorial proof of their result as well as a further refinement.

To do so, we consider statistics on an extension of the set $\mathcal{L}_{n,k}$.  We first recall two statistics on $\mathcal{L}_{n,k}$ from \cite{MSS}.
\begin{definition}\label{d1}
If $\rho \in \mathcal{L}_{n,k}$ and $i \in [n]$, then we say that $i$ is a record low of $\rho$ if there are no elements $j<i$ to the left of $i$ within its block in $\lambda$. \end{definition}
For example, if $\rho=\{4,1,5\},\{8,3,6,2\},\{7\} \in \mathcal{L}_{8,3}$, then the elements $4$ and $1$ are record lows in the first block, $8$, $3$ and $2$ are record lows in the second, and $7$ is a record low in the third block for a total of six record lows altogether.  Note that the smallest element within a block as well as the left-most one are always record lows.
\begin{definition}\label{d2}
Given $\rho \in \mathcal{L}_{n,k}$, let $rec^*(\rho)$ denote the number of record lows of $\rho$ which are not themselves the smallest member of a block.  Let $nrec(\rho)$ denote the number of elements of $[n]$ which are not record lows of $\rho$.
\end{definition}

We now consider an extension of the set $\mathcal{L}_{n}=\cup_{k=0}^n\mathcal{L}_{n,k}$ consisting of all Lah distributions of size $n$.  We assume throughout the remainder of this section that the blocks within a member of $\mathcal{L}_{n}$ are arranged from left to right in increasing order according to the size of the smallest element.
\begin{definition}\label{d3}
Let us call an element $i$ special within $\rho \in \mathcal{L}_{n}$ if $i=1$ or if $i\geq 2$ and the following two conditions are satisfied:
\begin{item}
(1) $i$ is not the smallest element of a block of $\rho$, and

(2) all elements of  $[i-1]$ occur to the left of $i$ in a left-to-right scan of the contents of the blocks of $\rho$.
\end{item}
\end{definition}

\begin{definition}\label{d4}
Let $\mathcal{L}_{n}^*$ denote the set obtained from $\mathcal{L}_{n}$ by circling some subset (possibly empty) of the special elements within each member of $\mathcal{L}_{n}$, where the element $1$ can be circled only if it starts a block. \end{definition}

We will refer to the members of $\mathcal{L}_n^*$ as \emph{extended Lah distributions}. Let us denote a circled element $i$ within a block by $\circled{i}$. Note that conditions (1) and (2) in Definition \ref{d3} above imply that $\circled{i}$ cannot start a block if $i \geq 2$, with $\circled{1}$ always starting a block when it occurs.  Otherwise, elements may be ordered in any way within blocks such that conditions (1) and (2) are met regarding the positions of any circled elements.

We consider a refinement of the set $\mathcal{L}_n^*$.

\begin{definition}\label{d5}
By a true block within $\lambda\in \mathcal{L}_n^*$, we will mean a block not containing $\circled{1}$ when $1$ is circled and any block of $\lambda$ when it is not.
Let $\mathcal{L}_{n,k}^*$ denote the subset of $\mathcal{L}_{n}^*$ whose members contain exactly $k$ true blocks.
\end{definition}
For example, we have $$\lambda =\{\circled{1},3,\circled{2}\}, \{4,\circled{5},7\}, \{13,6,8,\circled{9}\}, \{12,11,10,14,\circled{15}\} \in \mathcal{L}_{15,3}^*,$$
where the first block is not true and the elements $8$ and $14$ are special but not circled.  By the definitions, we have $\mathcal{L}_n^*=\cup_{k=0}^n \mathcal{L}_{n,k}^*$.  Furthermore, $|\mathcal{L}_{n,0}^*|=n!$ since such partitions contain a single block starting with $\circled{1}$, with each element $i$, $i \geq 2$, either uncircled and directly following some member (possibly circled) of $[i-1]$  or circled and occurring in the $i$-th position from the left.  Note that $\mathcal{L}_{n,k}$ corresponds to the subset of $\mathcal{L}_{n,k}^*$ in which no special elements are circled.

We now extend the $\text{rec}^{*}$ and $\text{nrec}$ statistics to $\mathcal{L}_{n,k}^*$ by disregarding any circled members of $[n]$ in true blocks when obtaining the values for $\text{rec}^{*}$ and $\text{nrec}$.  That is, within each true block, we only consider the sublist comprising the uncircled elements in the block when determining the contribution of that block towards $\text{rec}^*$ or $\text{nrec}$.  For a block containing $\circled{1}$, we consider the sublist of uncircled elements with $1$ written at the front of it.  That is, for a block containing $\circled{1}$, we require that it contribute zero towards $\text{rec}^*$ and contribute $p$ towards the $\text{nrec}$ value, where $p$ denotes the number of uncircled elements in the block.

Note that whether or not an uncircled element within a true block is a record low is independent of any circled elements belonging to the block since special elements greater than one cannot start blocks.

Given $\lambda \in \mathcal{L}_{n,k}^*$, let $\text{circ}(\lambda)$ denote the number of circled elements of $\lambda$.
\begin{definition}\label{d6}
If $\lambda \in \mathcal{L}_{n,k}^*$, then define the weight of $\lambda$ by
$$w(\lambda)=\alpha^{\text{nrec}(\lambda)}\beta^{\text{rec}^*(\lambda)}r^{\text{circ}(\lambda)}.$$
\end{definition}
For example, if $\lambda \in \mathcal{L}_{15,3}^*$ is as above, then we have $\text{nrec}(\lambda)=4$ as there is an uncircled element that is not a record low
in each block (namely, $3$, $7$, $8$ and $14$), $\text{rec}^*(\lambda)=3$ (corresponding to $13$, $12$ and $11$) and $\text{circ}(\lambda)=5$, which implies $w(\lambda)=\alpha^4\beta^3r^5$.

We now recall the generalized Stirling numbers of Hsu and Shiue \cite{HS}.  Given a non-negative integer $k$ and increment $\theta$, let
$$(x)^{(k,\theta)}=x(x-\theta)\cdots (x-k\theta+\theta), \qquad k \geq 1,$$
with $(x)^{(0,\theta)}=1$.  If $n$, $k$ and $r$ are non-negative integers, then the numbers $S(n,k;\alpha,\beta,r)$ are defined as the connection constants in the polynomial identities
\begin{equation}\label{genrec}
(x)^{(n,-\alpha)}=\sum_{k=0}^n S(n,k;\alpha,\beta,r)(x-r)^{(k,\beta)}.
\end{equation}
Note that $S(n,k;\alpha,\beta,r)$ reduces to the Stirling number of the second kind $S(n,k)$ when $\alpha=r=0$ and $\beta=1$ and to the Lah number $L(n,k)$ when $r=0$ and $-\alpha=\beta=1$.

The following lemma provides a combinatorial interpretation of $S(n,k;\alpha,\beta,r)$ as a weighted sum.

\begin{lemma}\label{genl1}
If $n,k,r \geq 0$, then
\begin{equation}\label{genl1e1}
S(n,k;\alpha,\beta,r)=\sum_{\lambda \in \mathcal{L}_{n,k}^*} w(\lambda).
\end{equation}
\end{lemma}
\begin{proof}
Using \eqref{genrec}, one can show
\begin{equation}\label{genl1e2}
S(n,k;\alpha,\beta,r)=S(n-1,k-1;\alpha,\beta,r)+(\alpha(n-1)+\beta k + r)S(n-1,k;\alpha,\beta,r), \qquad n\geq 1, ~k \geq 0,
\end{equation}
with $S(n,k;\alpha,\beta,r)=0$ if $0 \leq n <k$ and $1$ if $n=k=0$.  Formula \eqref{genl1e1} now follows by an inductive argument using \eqref{genl1e2} upon considering the position of the element $n$ within a member $\mathcal{L}_{n,k}^*$ for $n >1$ (the formula is easily seen to hold for $n=0,1$).  Note that if $n$ occupies its own (true) block, then there are $S(n-1,k-1;\alpha,\beta,r)$ possibilities.  On the other hand, if $n$ occupies any block containing at least one member of $[n-1]$, then $n$ adds one to the $\text{nrec}$ value if it directly follows some member of $[n-1]$ within the block (and itself is not circled, though its predecessor may be circled) and adds one to the $\text{rec}^*$ value if it comes at the beginning of a true block containing at least one member of $[n-1]$ (since in that case it would be a record low that is not minimal).  The final option, which increases $\text{circ}$ by one, but leaves the other two statistics unchanged as well as the number of true blocks, is for $n$ to be circled, which implies that it must occur as the last element of the right-most block, with this block not a singleton.  Combining these last three cases gives the second term on the right-hand side of \eqref{genl1e2} and completes the proof.
\end{proof}

Define the generalized Bell polynomial \cite{Xu} by
$$B_{n;\alpha,\beta,r}(x)=\sum_{k=0}^n S(n,k;\alpha,\beta,r)x^k, \qquad n \geq 0.$$

We now can give a combinatorial proof of the recurrence for $B_{n;\alpha,\beta,r}(x)$ shown in \cite{Xu}.

\begin{theorem}\label{t5}
If $m,n,r \geq 0$, then
\begin{equation}\label{t5e1}
S(m+n,k;\alpha,\beta,r)=\sum_{i=0}^n \sum_{j=0}^m \binom{n}{i}S(m,j;\alpha,\beta,r)S(i,k-j;\alpha,\beta,r)\prod_{\ell=0}^{n-i-1}((m+\ell)\alpha+j\beta)
\end{equation}
and
\begin{equation}\label{t5e2}
B_{m+n;\alpha,\beta,r}(x)=\sum_{i=0}^n \sum_{j=0}^m \binom{n}{i}x^jS(m,j;\alpha,\beta,r)B_{i;\alpha,\beta,r}(x)\prod_{\ell=0}^{n-i-1}((m+\ell)\alpha+j\beta).
\end{equation}
\end{theorem}
\begin{proof}
We first show \eqref{t5e1}.  The recurrence is easily verified if $m=0$ or $n=0$, so we will assume $m,n \geq 1$.  Let $\mathcal{L}_{m+n,k}^*(i,j)$ denote the subset of $\mathcal{L}_{m+n,k}^*$ whose members $\lambda$ satisfy the following two properties:

(i) the subpartition of $\lambda$ consisting of the elements of $[m]$ forms a member of $\mathcal{L}_{m,j}^*$, and

(ii) there are exactly $n-i$ uncircled elements of $H=[m+1,m+n]$ occurring in the blocks of the subpartition described in (i) and to the left of any circled elements of $H$ occurring within these blocks.

Note that any circled elements of $H$ lying within the first $j$ true blocks of $\lambda$ (or in the block containing $\circled{1}$, if it occurs) must belong to the $j$-th true block if $j \geq 1$ or to the block containing $\circled{1}$ if $j=0$.  We will show that the $i,j$ term in the sum on the right-hand side of \eqref{t5e1} gives $\sum_{\lambda \in \mathcal{L}_{m+n,k}^*(i,j)}w(\lambda)$, whence the result follows from summing over all $i$ and $j$ and appealing to Lemma \ref{genl1}.

To do so, let us express $\lambda \in \mathcal{L}_{m+n,k}^*(i,j)$ as $\lambda=(\sigma,\tau)$, where $\sigma$ and $\tau$ are smaller extended Lah distributions defined as follows.  The distribution $\sigma$ is obtained by taking the subpartition of $\lambda$ of size $m+n-i$ consisting of the elements of $[m]$, together with the $n-i$ elements of $H$ described in condition (ii) above.  The distribution $\tau$ consists of the remaining elements of $[m+n]-\sigma$ and is obtained as follows.  If the $j$-th true block of $\lambda$ (which we take to be the block containing $\circled{1}$ if $j=0$) contains no circled elements of $H$, then $\tau$ is obtained by taking the last $k-j$ blocks of $\lambda$, in which case $\tau$ is either empty or all of its blocks are true.  Otherwise, suppose that the $j$-th block of $\lambda$ contains at least one circled element of $H$, and let $c$ denote the smallest such element.  In this case, $\tau$ is non-empty and is obtained by letting its first block comprise all of the elements (in order) to the right of and including $\circled{c}$ within the $j$-th block of $\lambda$ and letting the $k-j$ remaining blocks of $\lambda$ form the true blocks of $\tau$.  Note that $c$ is the smallest element of $\tau$ since all letters to the right of $c$ in $\lambda$ are larger than $c$, since $c$ is circled.

Next observe that given $(\sigma,\tau)$ as defined above, one can reconstruct $\lambda$, and thus the mapping $\lambda \mapsto (\sigma,\tau)$ is a bijection.  Furthermore, all true blocks of $\tau$ remain true when $\tau$ is combined with $\sigma$ to form $\lambda$, with the values of the $\text{nrec}$, $\text{rec}^*$ and $\text{circ}$ statistics on $\sigma$ and $\tau$ adding to yield the respective values for $\lambda$.  Note concerning $\text{rec}^*$ that an uncircled element to the right of $\circled{c}$ in the $j$-th block of $\lambda$, where $c$ is as described above, cannot contribute to the $\text{rec}^*$ value of $\lambda$ since there is at least one element of $[m]$ to the left of $\circled{c}$ in this block.  In this case, then the first block of the corresponding $\tau$ is not true and thus contributes zero to the $\text{rec}^*$ value of $\tau$.

Now observe that there are $\binom{n}{i}S(m,j;\alpha,\beta,r)\prod_{\ell=0}^{n-i-1}((m+\ell)\alpha+j\beta)$
possibilities for $\sigma$.  To see this, simply arrange the elements of $[m]$ according to some $\rho\in\mathcal{L}_{m,j}^*$, select $n-i$ elements of $H$ in $\binom{n}{i}$ ways, and then add them as uncircled elements to $\rho$, starting with the smallest.  Note that an added element of $H$ to $\rho$ may follow any member (possibly circled) of $[m]$, or any previously added element, or come at the beginning of any of the true blocks of $\rho$.

Upon relabeling elements as members in $[i]$, we see that there are $S(i,k-j;\alpha,\beta,r)$ possibilities for $\tau$, which follows from how the distribution $\tau$ was obtained from $\lambda$.  Since $\tau$ may be chosen to be any member of $\mathcal{L}_{i,k-j}^*$ once $\sigma$ is known, it follows by multiplication that the weight of all members of $\mathcal{L}_{m+n,k}^*(i,j)$ is given by the $i,j$ term of the sum, as desired, which finishes the proof of \eqref{t5e1}.  Summing over $k$, or allowing $\tau$ to contain any number of true blocks in the preceding argument, gives the $x=1$ case of \eqref{t5e2}.  Adding a variable $x$ marking the number of true blocks yields \eqref{t5e2} in general and completes the proof.
\end{proof}

We remark that formula \eqref{t5e2} appears as \cite[Theorem 4]{Xu} with $\alpha$ replaced by $-\alpha$, while \eqref{t5e1} does not seem to have been previously noted.

\textbf{Acknowledgement:} The author wishes to thank Toufik Mansour and Matthias Schork for useful discussions and pointing out to him the paper of Mez\H{o} \cite{Me}.

\end{document}